\documentclass{article}
\usepackage{amssymb}
\usepackage{amsfonts}
\usepackage{amsmath}

\newtheorem{theorem}{Theorem}

\newenvironment{proof}[1][Proof]{\noindent\textbf{#1.} }{\ \rule{0.5em}{0.5em}}
\input{tcilatex}
\begin{document}

\title{Sampling Theorem Associated with \\
q-Dirac System}
\author{Fatma H\i ra \\
%EndAName
Hitit University, Arts and Science Faculty, \\
Department of Mathematics, 19030, \c{C}orum,Turkey }
\maketitle

\begin{abstract}
This paper deals with $q-$analogue of sampling theory associated with $q-$%
Dirac system. We derive sampling representation for transform whose kernel
is a solution of this $q-$Dirac system. As a special case, three examples
are given.
\end{abstract}

\section{Introduction}

Consider the following $q-$Dirac system%
\begin{equation}
\left\{ 
\begin{array}{l}
-\dfrac{1}{q}D_{q^{-1}}y_{2}+p\left( x\right) y_{1}=\lambda y_{1}, \\ 
D_{q}y_{1}+r\left( x\right) y_{2}=\lambda y_{2},%
\end{array}%
\right.  \tag{1.1}
\end{equation}%
\begin{equation}
k_{11}y_{1}\left( 0\right) +k_{12}y_{2}\left( 0\right) =0,  \tag{1.2}
\end{equation}%
\begin{equation}
k_{21}y_{1}\left( a\right) +k_{22}y_{2}\left( aq^{-1}\right) =0,  \tag{1.3}
\end{equation}%
where $k_{ij}~\left( i,j=1,2\right) $ are real numbers, $\lambda $ is a
complex eigenvalue parameter, $y\left( x\right) =\left( 
\begin{array}{c}
y_{1}\left( x\right) \\ 
y_{2}\left( x\right)%
\end{array}%
\right) ,$ $p\left( x\right) $ and $r\left( x\right) $ are real-valued
functions defined on $\left[ 0,a\right] $ and continuous at zero and $%
p\left( x\right) $, $r\left( x\right) \in L_{q}^{1}\left( 0,a\right) ~\left( 
\text{see }\left[ 1,2\right] \right) .$

The papers in $q-$Dirac system are few, see $\left[ 1-3\right] .$ However
sampling theories associated with $q-$Dirac system do not exist as far as we
know. So that we will construct a $q-$analogue of sampling theorem for $q-$%
Dirac system (1.1)-(1.3), building on recent results in $\left[ 1,2\right] .$
To achieve our aim we will briefly give the spectral analysis of the problem
(1.1)-(1.3). Then we derive sampling theorem using solution. In the last
section we give three examples illustrating the obtained results.

\section{Notations and Preliminaries}

We state the $q-$notations and results which will be needed for the
derivation of the sampling theorem. Throughout this paper $q$ is a positive
number with $0<q<1$.

$A$ set $A\subseteq 
%TCIMACRO{\U{211d} }%
%BeginExpansion
\mathbb{R}
%EndExpansion
$ is called $q$-geometric if, for every $x\in A,~qx\in A$. Let $f$ be a real
or complex-valued function defined on a $q$-geometric set $A$. The $q$%
-difference operator is defined by%
\begin{equation}
D_{q}f\left( x\right) :=\frac{f\left( x\right) -f\left( qx\right) }{x\left(
1-q\right) },~x\neq 0.  \tag{2.1}
\end{equation}%
If $0\in A$, the $q$-derivative at zero is defined to be%
\begin{equation}
D_{q}f\left( 0\right) :=\lim_{n\rightarrow \infty }\frac{f\left(
xq^{n}\right) -f\left( 0\right) }{xq^{n}},~x\in A,  \tag{2.2}
\end{equation}%
if the limit exists and does not depend on $x$. Also, for $x\in A$, $%
D_{q^{-1}}$ is defined to be%
\begin{equation}
D_{q^{-1}}f\left( x\right) :=\left\{ 
\begin{array}{l}
\dfrac{f\left( x\right) -f\left( q^{-1}x\right) }{x\left( 1-q^{-1}\right) }%
,~x\in A\setminus \left\{ 0\right\} , \\ 
D_{q}f\left( 0\right) ,~\text{\ \ \ \ \ \ \ \ \ \ \ \ }x=0,%
\end{array}%
\right.  \tag{2.3}
\end{equation}%
provided that $D_{q}f\left( 0\right) $ exists. The following relation can be
verified directly from the definition%
\begin{equation}
D_{q^{-1}}f\left( x\right) =\left( D_{q}f\right) \left( xq^{-1}\right) . 
\tag{2.4}
\end{equation}%
A right inverse, $q$-integration, of the $q$-difference operator $D_{q}$ is
defined by Jackson $\left[ 4\right] $ as%
\begin{equation}
\int\limits_{0}^{x}f\left( t\right) d_{q}t:=x\left( 1-q\right)
\sum\limits_{n=0}^{\infty }q^{n}f\left( xq^{n}\right) ,~x\in A,  \tag{2.5}
\end{equation}%
provided that the series converges. A $q$-analog of the fundamental theorem
of calculus is given by%
\begin{equation}
D_{q}\int\limits_{0}^{x}f\left( t\right) d_{q}t=f\left( x\right) ,\text{ }%
\int\limits_{0}^{x}D_{q}f\left( t\right) d_{q}t=f\left( x\right)
-\lim_{n\rightarrow \infty }f\left( xq^{n}\right) ,\text{\ }  \tag{2.6}
\end{equation}%
where $\underset{n\rightarrow \infty }{\lim }f\left( xq^{n}\right) $ can be
replaced by $f\left( 0\right) $ if $f$ is $q$-regular at zero, that is, if $%
\underset{n\rightarrow \infty }{\lim }f\left( xq^{n}\right) =f\left(
0\right) $, for all $x\in A.$ Throughout this paper, we deal only with
functions $q$-regular at zero.

The $q$-type product formula is given by%
\begin{equation}
D_{q}\left( fg\right) \left( x\right) =g\left( x\right) D_{q}f\left(
x\right) +f\left( qx\right) D_{q}g\left( x\right) ,  \tag{2.7}
\end{equation}%
and hence the $q$-integration by parts is given by%
\begin{equation}
\int\limits_{0}^{a}g\left( x\right) D_{q}f\left( x\right) d_{q}x=\left(
fg\right) \left( a\right) -\left( fg\right) \left( 0\right)
-\int\limits_{0}^{a}D_{q}g\left( x\right) f\left( qx\right) d_{q}x, 
\tag{2.8}
\end{equation}%
where $f$ and $g$ are $q$-regular at zero.

For more results and properties in $q$-calculus, readers are referred to the
recent works $\left[ 5-8\right] $.

The basic trigonometric functions $\cos \left( z;q\right) $ and $\sin \left(
z;q\right) $ are defined on $%
%TCIMACRO{\U{2102} }%
%BeginExpansion
\mathbb{C}
%EndExpansion
$ by%
\begin{equation}
\cos \left( z;q\right) :=\sum\limits_{n=0}^{\infty }\frac{\left( -1\right)
^{n}q^{n^{2}}\left( z\left( 1-q\right) \right) ^{2n}}{\left( q;q\right) _{2n}%
},  \tag{2.9}
\end{equation}%
\begin{equation}
\sin \left( z;q\right) :=\sum\limits_{n=0}^{\infty }\frac{\left( -1\right)
^{n}q^{n\left( n+1\right) }\left( z\left( 1-q\right) \right) ^{2n+1}}{\left(
q;q\right) _{2n+1}},  \tag{2.10}
\end{equation}%
and they are $q$-analogs of the cosine and sine functions. $\cos \left(
.;q\right) $ and $\sin \left( .;q\right) $ have only real and simple zeros $%
\left\{ \pm x_{m}\right\} _{m=1}^{\infty }$ and $\left\{ 0,\pm y_{m}\right\}
_{m=1}^{\infty },$respectively, where $x_{m},~y_{m}>0,m\geqslant 1$ and 
\begin{equation}
x_{m}=\left( 1-q\right) ^{-1}q^{-m+1/2+\varepsilon _{m}\left( 1/2\right) }%
\text{ if }q^{3}<\left( 1-q^{2}\right) ^{2},  \tag{2.11}
\end{equation}%
\begin{equation}
y_{m}=\left( 1-q\right) ^{-1}q^{-m+\varepsilon _{m}\left( -1/2\right) }\text{
if }q<\left( 1-q^{2}\right) ^{2}.  \tag{2.12}
\end{equation}%
Moreover, for any $q\in \left( 0,1\right) ,$ (2.11) and (2.12) hold for
sufficiently large $m,$ cf. $\left[ 5,9-11\right] .$

Let $L_{q}^{2}\left( 0,a\right) $ be the space of all complex valued
functions defined on $\left[ 0,a\right] $ such that%
\begin{equation}
\left\Vert f\right\Vert :=\left( \int\limits_{0}^{a}\left\vert f\left(
x\right) \right\vert ^{2}d_{q}x\right) ^{1\backslash 2}<\infty .  \tag{2.13}
\end{equation}%
The space $L_{q}^{2}\left( 0,a\right) $ is a separable Hilbert space with
the inner product $\left( \text{see }\left[ 12\right] \right) $%
\begin{equation}
\left\langle f,g\right\rangle :=\int\limits_{0}^{a}f\left( x\right) 
\overline{g\left( x\right) }d_{q}x,~f,g\in L_{q}^{2}\left( 0,a\right) . 
\tag{2.14}
\end{equation}

Let $H_{q}$ be the Hilbert space 
\begin{equation*}
H_{q}:=\left\{ y\left( x\right) =\left( 
\begin{array}{c}
y_{1}\left( x\right) \\ 
y_{2}\left( x\right)%
\end{array}%
\right) ,~y_{1}\left( x\right) ,y_{2}\left( x\right) \in L_{q}^{2}\left(
0,a\right) \right\} .
\end{equation*}%
The inner product of $H_{q}$ is defined by%
\begin{equation}
\left\langle y\left( .\right) ,z\left( .\right) \right\rangle
_{H_{q}}:=\int\limits_{0}^{a}y^{\top }\left( x\right) z\left( x\right)
d_{q}x,  \tag{2.15}
\end{equation}%
where $\top $ denotes the matrix transpose, $y\left( x\right) =\left( 
\begin{array}{c}
y_{1}\left( x\right) \\ 
y_{2}\left( x\right)%
\end{array}%
\right) ,$ $z\left( x\right) =\left( 
\begin{array}{c}
z_{1}\left( x\right) \\ 
z_{2}\left( x\right)%
\end{array}%
\right) \in H_{q},~y_{i}\left( .\right) ,~z_{i}\left( .\right) \in
L_{q}^{2}\left( 0,a\right) $ $\left( i=1,2\right) .$

It is known $\left[ 2\right] $ that the problem (1.1)-(1.3) has a countable
number of eigenvalues $\left\{ \lambda _{n}\right\} _{n=-\infty }^{\infty }$
which are real and simple, and to every eigenvalue $\lambda _{n},$ there
corresponds a vector-valued eigenfunction $y_{n}^{\top }\left( x,\lambda
_{n}\right) =\left( y_{n,1}\left( x,\lambda _{n}\right) ,y_{n,2}\left(
x,\lambda _{n}\right) \right) .$ Moreover, vector-valued eigenfunctions
belonging to different eigenvalues are orthogonal, i.e.,%
\begin{equation*}
\begin{array}{l}
\int\limits_{0}^{a}y_{n}^{\top }\left( x,\lambda _{n}\right) y_{m}\left(
x,\lambda _{m}\right) d_{q}x \\ 
=\int\limits_{0}^{a}\left\{ y_{n,1}\left( x,\lambda _{n}\right)
y_{m,1}\left( x,\lambda _{m}\right) +y_{n,2}\left( x,\lambda _{n}\right)
y_{m,2}\left( x,\lambda _{m}\right) \right\} d_{q}x=0,~\ \text{for }\lambda
_{n}\neq \lambda _{m}.%
\end{array}%
\end{equation*}

Let $y_{1}\left( x,\lambda _{1}\right) =\left( 
\begin{array}{c}
y_{11}\left( x,\lambda _{1}\right) \\ 
y_{12}\left( x,\lambda _{1}\right)%
\end{array}%
\right) $ and $y_{2}\left( x,\lambda _{2}\right) =\left( 
\begin{array}{c}
y_{21}\left( x,\lambda _{2}\right) \\ 
y_{22}\left( x,\lambda _{2}\right)%
\end{array}%
\right) $ be solutions of (1.1): hence%
\begin{equation}
\left\{ 
\begin{array}{l}
-\frac{1}{q}D_{q^{-1}}y_{12}+\left\{ p\left( x\right) -\lambda _{1}\right\}
y_{11}=0, \\ 
D_{q}y_{11}+\left\{ r\left( x\right) -\lambda _{1}\right\} y_{12}=0,%
\end{array}%
\right.  \tag{2.16}
\end{equation}%
and%
\begin{equation}
\left\{ 
\begin{array}{l}
-\frac{1}{q}D_{q^{-1}}y_{22}+\left\{ p\left( x\right) -\lambda _{2}\right\}
y_{21}=0, \\ 
D_{q}y_{21}+\left\{ r\left( x\right) -\lambda _{2}\right\} y_{22}=0.%
\end{array}%
\right.  \tag{2.17}
\end{equation}

Multiplying (2.16) by $y_{21}$ and $y_{22}$ and (2.17) by $-y_{11}$ and $%
-y_{22}$ respectively, and adding them together also using the formula (2.4)
we obtain%
\begin{equation}
\begin{array}{l}
D_{q}\left\{ y_{11}\left( x,\lambda _{1}\right) y_{22}\left( xq^{-1},\lambda
_{2}\right) -y_{12}\left( xq^{-1},\lambda _{1}\right) y_{21}\left( x,\lambda
_{2}\right) \right\} \\ 
=\left( \lambda _{1}-\lambda _{2}\right) \left\{ y_{11}\left( x,\lambda
_{1}\right) y_{21}\left( x,\lambda _{2}\right) +y_{12}\left( x,\lambda
_{1}\right) y_{22}\left( x,\lambda _{2}\right) \right\} .%
\end{array}
\tag{2.18}
\end{equation}

Let $y\left( x\right) =\left( 
\begin{array}{c}
y_{1}\left( x\right) \\ 
y_{2}\left( x\right)%
\end{array}%
\right) ,~z\left( x\right) =\left( 
\begin{array}{c}
z_{1}\left( x\right) \\ 
z_{2}\left( x\right)%
\end{array}%
\right) \in H_{q}.$Then the Wronskian of $y\left( x\right) $ and $z\left(
x\right) $ is defined by%
\begin{equation}
W\left( y,z\right) \left( x\right) :=y_{1}\left( x\right) z_{2}\left(
xq^{-1}\right) -z_{1}\left( x\right) y_{2}\left( xq^{-1}\right) .  \tag{2.19}
\end{equation}

Let us consider the next initial value problem

\begin{equation}
\left\{ 
\begin{array}{l}
-\frac{1}{q}D_{q^{-1}}y_{2}+p\left( x\right) y_{1}=\lambda y_{1}, \\ 
D_{q}y_{1}+r\left( x\right) y_{2}=\lambda y_{2},%
\end{array}%
\right.  \tag{2.20}
\end{equation}

\begin{equation}
y_{1}\left( 0\right) =k_{12},~\ \ y_{2}\left( 0\right) =-k_{11}.  \tag{2.21}
\end{equation}

By virtue of Theorem 1 in $\left[ 1\right] ,$ this problem has a unique
solution $\phi \left( x,\lambda \right) =\left( 
\begin{array}{c}
\phi _{1}\left( x,\lambda \right) \\ 
\phi _{2}\left( x,\lambda \right)%
\end{array}%
\right) $ . It is obvious that $\phi \left( x,\lambda \right) $ satisfies
the boundary condition (1.2) and this function is uniformly bounded on the
subsets of the form $\left[ 0,a\right] \times \Omega $ where $\Omega \subset 
%TCIMACRO{\U{2102} }%
%BeginExpansion
\mathbb{C}
%EndExpansion
$ is compact. The proof is similar to the one in the proof of Lemma 3.1 in $%
\left[ 13\right] .$ To find the eigenvalues of the $q-$Dirac system
(1.1)-(1.3) we have to insert this function into the boundary condition
(1.3) and find the roots of the obtained equation. So, putting the function $%
\phi \left( x,\lambda \right) $ into the boundary condition (1.3) we get the
following equation whose zeros are the eigenvalues of the $q-$Dirac system
(1.1)-(1.3)%
\begin{equation}
\omega \left( \lambda \right) =k_{21}\phi _{1}\left( a,\lambda \right)
+k_{22}\phi _{2}\left( aq^{-1},\lambda \right) .  \tag{2.22}
\end{equation}

It is also known that if $\left\{ \phi _{n}\left( .\right) \right\}
_{n=-\infty }^{\infty }$ denotes a set of vector-valued eigenfunctions
corresponding $\left\{ \lambda _{n}\right\} _{n=-\infty }^{\infty },$ then $%
\left\{ \phi _{n}\left( .\right) \right\} _{n=-\infty }^{\infty }$ is a
complete orthogonal set of $H_{q.}$ For more details about how to obtain the
solutions and the eigenvalues for $q-$Dirac system see $\left[ 1,2\right] ,$
similar to the classical case of Dirac system $\left[ 14\right] $ and $q-$%
Sturm-Liouville problems $\left[ 15,16\right] .$

\section{The Sampling Theory}

The WKS (Whittaker-Kotel'nikov-Shannon) $\left[ 17-19\right] $ sampling
theorem has been generalized in many different ways. The connection between
the WKS sampling theorem and boundary value problems was first observed by
Weiss $\left[ 20\right] $ and followed by Kramer $\left[ 21\right] .$ In $%
\left[ 22\right] $, sampling theorem is introduced where sampling
representations are derived for integral transforms whose kernels are
solutions of one-dimensional regular Dirac systems. In recent years, the
connection between sampling theorems and $q-$boundary value problems has
been the focus of many research papers. In $\left[ 12,23\right] ,$ $q-$%
versions of the classical sampling theorem of WKS as well as Kramer's
analytic theorem were introduced. These results were extended to $q-$%
Sturm-Liouville problems in $\left[ 13,24\right] $, singular
q-Sturm-Liouville problem in $\left[ 25\right] $ and the $q,\omega -$%
Hahn-Sturm-Liouville problem in $\left[ 26\right] .$

In this section, we state and prove $q-$analogue of sampling theorem
associated with $q-$Dirac system (1.1)-(1.3), inspired by the classical case 
$\left[ 22\right] .$

\begin{theorem}
Let $f\left( x\right) =\left( 
\begin{array}{c}
f_{1}\left( x\right) \\ 
f_{2}\left( x\right)%
\end{array}%
\right) \in H_{q}$ and $F\left( \lambda \right) $ be the $q-$type transform%
\begin{equation}
F\left( \lambda \right) =\int\limits_{0}^{a}f^{\top }\left( x\right) \phi
\left( x,\lambda \right) d_{q}x,~~\lambda \in 
%TCIMACRO{\U{2102} }%
%BeginExpansion
\mathbb{C}
%EndExpansion
,  \tag{3.1}
\end{equation}%
where $\phi \left( x,\lambda \right) $ is the solution defined above. Then $%
F\left( \lambda \right) $ is an entire function that can be reconstructed
using its values at the points $\left\{ \lambda _{n}\right\} _{n=-\infty
}^{\infty }$ by means of the sampling form%
\begin{equation}
F\left( \lambda \right) =\sum\limits_{n=-\infty }^{\infty }F\left( \lambda
_{n}\right) \frac{\omega \left( \lambda \right) }{\left( \lambda -\lambda
_{n}\right) \omega ^{\prime }\left( \lambda _{n}\right) },  \tag{3.2}
\end{equation}%
where $\omega \left( \lambda \right) $ is defined in (2.22). The series
(3.2) converges absolutely on $%
%TCIMACRO{\U{2102} }%
%BeginExpansion
\mathbb{C}
%EndExpansion
$ and uniformly on compact subsets of $%
%TCIMACRO{\U{2102} }%
%BeginExpansion
\mathbb{C}
%EndExpansion
.$
\end{theorem}

\begin{proof}
Since $\phi \left( x,\lambda \right) $ is in $H_{q}$ for any $\lambda ,$ we
have%
\begin{equation}
\phi \left( x,\lambda \right) =\sum\limits_{n=-\infty }^{\infty }\widehat{%
\phi }_{n}\frac{\phi _{n}\left( x\right) }{\left\Vert \phi _{n}\right\Vert
_{H_{q}}^{2}},  \tag{3.3}
\end{equation}%
where%
\begin{equation}
\begin{array}{l}
\widehat{\phi }_{n}=\int\limits_{0}^{a}\phi ^{\top }\left( x,\lambda \right)
\phi _{n}\left( x\right) d_{q}x \\ 
\text{ \ \ }=\int\limits_{0}^{a}\left\{ \phi _{1}\left( x,\lambda \right)
\phi _{n,1}\left( x\right) +\phi _{2}\left( x,\lambda \right) \phi
_{n,2}\left( x\right) \right\} d_{q}x,%
\end{array}
\tag{3.4}
\end{equation}%
$\phi ^{\top }\left( x,\lambda \right) =\left( \phi _{1}\left( x,\lambda
\right) ,~\phi _{2}\left( x,\lambda \right) \right) $ and $\phi _{n}^{\top
}\left( x\right) =\left( \phi _{n,1}\left( x\right) ,~\phi _{n,2}\left(
x\right) \right) $ is the vector-valued eigenfunction corresponding to the
eigenvalue $\lambda _{n}.$

Since $f$ is in $H_{q}$, it has the Fourier expansion%
\begin{equation}
f\left( x\right) =\sum\limits_{n=-\infty }^{\infty }\widehat{f}_{n}\frac{%
\phi _{n}\left( x\right) }{\left\Vert \phi _{n}\right\Vert _{H_{q}}^{2}}, 
\tag{3.5}
\end{equation}%
where%
\begin{equation}
\begin{array}{l}
\widehat{f}_{n}=\int\limits_{0}^{a}f^{\top }\left( x\right) \phi _{n}\left(
x\right) d_{q}x \\ 
\text{ \ \ }=\int\limits_{0}^{a}\left\{ f_{1}\left( x\right) \phi
_{n,1}\left( x\right) +f_{2}\left( x\right) \phi _{n,2}\left( x\right)
\right\} d_{q}x.%
\end{array}
\tag{3.6}
\end{equation}%
In view of Parseval's relation and definition (3.1), we obtain%
\begin{equation}
F\left( \lambda \right) =\sum\limits_{n=-\infty }^{\infty }F\left( \lambda
_{n}\right) \frac{\widehat{\phi }_{n}}{\left\Vert \phi _{n}\right\Vert
_{H_{q}}^{2}}.  \tag{3.7}
\end{equation}%
Let $\lambda \in 
%TCIMACRO{\U{2102} }%
%BeginExpansion
\mathbb{C}
%EndExpansion
,$ $\lambda \neq \lambda _{n}$ and $n\in 
%TCIMACRO{\U{2115} }%
%BeginExpansion
\mathbb{N}
%EndExpansion
$ be fixed. From relation (2.18), with $y_{11}\left( x\right) =\phi
_{1}\left( x,\lambda \right) ,y_{12}\left( x\right) =\phi _{2}\left(
x,\lambda \right) ~$and $y_{21}\left( x\right) =\phi _{n,1}\left( x\right) ,$
$y_{22}\left( x\right) =\phi _{n,2}\left( x\right) ,$ we obtain%
\begin{equation}
\begin{array}{l}
\left( \lambda -\lambda _{n}\right) \int\limits_{0}^{a}\left\{ \phi
_{1}\left( x,\lambda \right) \phi _{n,1}\left( x\right) +\phi _{2}\left(
x,\lambda \right) \phi _{n,2}\left( x\right) \right\} d_{q}x \\ 
=\left. W\left( \phi \left( .,\lambda \right) ,\phi _{n}\left( .\right)
\right) \right\vert _{x=a}-\left. W\left( \phi \left( .,\lambda \right)
,\phi _{n}\left( .\right) \right) \right\vert _{x=0}.%
\end{array}
\tag{3.8}
\end{equation}%
From (2.19) and the definition of $\phi \left( .,\lambda \right) ,$ we have%
\begin{equation}
\begin{array}{l}
\left( \lambda -\lambda _{n}\right) \int\limits_{0}^{a}\left\{ \phi
_{1}\left( x,\lambda \right) \phi _{n,1}\left( x\right) +\phi _{2}\left(
x,\lambda \right) \phi _{n,2}\left( x\right) \right\} d_{q}x \\ 
=\phi _{1}\left( a,\lambda \right) \phi _{n,2}\left( aq^{-1}\right) -\phi
_{n,1}\left( a\right) \phi _{2}\left( aq^{-1},\lambda \right) .%
\end{array}
\tag{3.9}
\end{equation}%
Assume that $k_{22}\neq 0.$ Since $\phi _{n}\left( .\right) $ is an
eigenfunction, then it satisfies (1.3). Hence%
\begin{equation}
\phi _{n,2}\left( aq^{-1}\right) =-\frac{k_{21}}{k_{22}}\phi _{n,1}\left(
a\right) .  \tag{3.10}
\end{equation}%
Substituting from (3.10) in (3.9), we obtain%
\begin{equation}
\begin{array}{l}
\left( \lambda -\lambda _{n}\right) \int\limits_{0}^{a}\left\{ \phi
_{1}\left( x,\lambda \right) \phi _{n,1}\left( x\right) +\phi _{2}\left(
x,\lambda \right) \phi _{n,2}\left( x\right) \right\} d_{q}x \\ 
=-\phi _{n,1}\left( a\right) \left\{ \dfrac{k_{21}}{k_{22}}\phi _{1}\left(
a,\lambda \right) +\phi _{2}\left( aq^{-1},\lambda \right) \right\} \\ 
=-\dfrac{\omega \left( \lambda \right) \phi _{n,1}\left( a\right) }{k_{22}}%
\end{array}
\tag{3.11}
\end{equation}%
provided that $k_{22}\neq 0.$ Similarly, we can show that

\begin{equation}
\begin{array}{l}
\left( \lambda -\lambda _{n}\right) \int\limits_{0}^{a}\left\{ \phi
_{1}\left( x,\lambda \right) \phi _{n,1}\left( x\right) +\phi _{2}\left(
x,\lambda \right) \phi _{n,2}\left( x\right) \right\} d_{q}x \\ 
=-\dfrac{\omega \left( \lambda \right) \phi _{n,2}\left( aq^{-1}\right) }{%
k_{21}}%
\end{array}
\tag{3.12}
\end{equation}%
provided that $k_{21}\neq 0.$ Differentiating with respect to $\lambda $ and
taking the limit as $\lambda \rightarrow \lambda _{n},$ we obtain%
\begin{eqnarray}
\left\Vert \phi _{n}\right\Vert _{H_{q}}^{2} &=&\int\limits_{0}^{a}\phi
_{n}^{\top }\left( x\right) \phi _{n}\left( x\right) d_{q}x  \notag \\
&=&-\dfrac{\omega ^{\prime }\left( \lambda _{n}\right) \phi _{n,1}\left(
a\right) }{k_{22}},  \TCItag{3.13} \\
&=&-\dfrac{\omega ^{\prime }\left( \lambda _{n}\right) \phi _{n,2}\left(
aq^{-1}\right) }{k_{21}}.  \TCItag{3.14}
\end{eqnarray}%
From (3.4), (3.11) and (3.13), we have for $k_{22}\neq 0,$%
\begin{equation}
\frac{\widehat{\phi }_{n}}{\left\Vert \phi _{n}\right\Vert _{H_{q}}^{2}}=%
\frac{\omega \left( \lambda \right) }{\left( \lambda -\lambda _{n}\right)
\omega ^{\prime }\left( \lambda _{n}\right) },  \tag{3.15}
\end{equation}%
and if $k_{21}\neq 0,$ we use (3.4), (3.12) and (3.14) to obtain the same
result. Therefore from (3.7) and (3.15) we get (3.2) when $\lambda $ is not
an eigenvalue. Now we investigate the convergence of (3.2). Using
Cauchy-Schwarz inequality for $\lambda \in 
%TCIMACRO{\U{2102} }%
%BeginExpansion
\mathbb{C}
%EndExpansion
.$%
\begin{equation}
\begin{array}{l}
\sum\limits_{k=-\infty }^{\infty }\left\vert F\left( \lambda _{k}\right) 
\dfrac{\omega \left( \lambda \right) }{\left( \lambda -\lambda _{k}\right)
\omega ^{\prime }\left( \lambda _{k}\right) }\right\vert
=\sum\limits_{k=-\infty }^{\infty }\left\vert \widehat{f}_{k}\dfrac{\widehat{%
\phi }_{k}}{\left\Vert \phi _{k}\right\Vert _{H_{q}}^{2}}\right\vert \\ 
\leq \left( \sum\limits_{k=-\infty }^{\infty }\left\vert \dfrac{\widehat{f}%
_{k}}{\left\Vert \phi _{k}\right\Vert _{H_{q}}}\right\vert ^{2}\right)
^{1\backslash 2}\left( \sum\limits_{k=-\infty }^{\infty }\left\vert \dfrac{%
\widehat{\phi }_{k}}{\left\Vert \phi _{k}\right\Vert _{H_{q}}}\right\vert
^{2}\right) ^{1\backslash 2}<\infty ,%
\end{array}
\tag{3.16}
\end{equation}%
since $f\left( .\right) ,~\phi \left( .,\lambda \right) \in H_{q},$ then the
two series in the right-hand side of (3.16) converge. Thus series (3.2)
converge absolutely on $%
%TCIMACRO{\U{2102} }%
%BeginExpansion
\mathbb{C}
%EndExpansion
.$ As for uniform convergence on compact subsets of $%
%TCIMACRO{\U{2102} }%
%BeginExpansion
\mathbb{C}
%EndExpansion
,$ let $\Omega _{M}:=\left\{ \lambda \in 
%TCIMACRO{\U{2102} }%
%BeginExpansion
\mathbb{C}
%EndExpansion
,\text{ }\left\vert \lambda \right\vert \leq M\right\} $ $M$ is a fixed
positive number. Let $\lambda \in \Omega _{M}$ and $N>0.$ Define $\Gamma
_{N}\left( \lambda \right) $ to be%
\begin{equation}
\Gamma _{N}\left( \lambda \right) =\left\vert F\left( \lambda \right)
-\sum\limits_{k=-N}^{N}F\left( \lambda _{k}\right) \dfrac{\omega \left(
\lambda \right) }{\left( \lambda -\lambda _{k}\right) \omega ^{\prime
}\left( \lambda _{k}\right) }\right\vert .  \tag{3.17}
\end{equation}%
By Cauchy-Schwarz inequality%
\begin{equation*}
\Gamma _{N}\left( \lambda \right) \leq \left\Vert \phi \left( .,\lambda
\right) \right\Vert _{H_{q}}\left( \sum\limits_{k=-N}^{N}\dfrac{\left\vert 
\widehat{f}_{k}\right\vert ^{2}}{\left\Vert \phi _{k}\right\Vert _{H_{q}}^{2}%
}\right) ^{1\backslash 2}.
\end{equation*}%
Since the function $\phi \left( .,\lambda \right) $ is uniformly bounded on
the subsets of $%
%TCIMACRO{\U{2102} }%
%BeginExpansion
\mathbb{C}
%EndExpansion
,$ we can find a positive constant $C_{\Omega }$ which is independent of $%
\lambda $ such that $\left\Vert \phi \left( .,\lambda \right) \right\Vert
_{H_{q}}\leq C_{\Omega },$ $\lambda \in \Omega _{M}.$ Thus%
\begin{equation*}
\Gamma _{N}\left( \lambda \right) \leq C_{\Omega }\left(
\sum\limits_{k=-N}^{N}\dfrac{\left\vert \widehat{f}_{k}\right\vert ^{2}}{%
\left\Vert \phi _{k}\right\Vert _{H_{q}}^{2}}\right) ^{1\backslash
2}\rightarrow 0\text{ as }N\rightarrow \infty .
\end{equation*}%
Hence (3.2) converges uniformly on compact subsets of $%
%TCIMACRO{\U{2102} }%
%BeginExpansion
\mathbb{C}
%EndExpansion
.$ Thus $F\left( \lambda \right) $ is an entire function and the proof is
complete.
\end{proof}

\section{Examples}

In this section we give three examples illustrating the sampling theorem of
the previous section.

\textbf{Example 1.} Consider $q-$Dirac system (1.1)-(1.3) in which $p\left(
x\right) =0=r\left( x\right) :$%
\begin{equation}
\left\{ 
\begin{array}{l}
-\dfrac{1}{q}D_{q^{-1}}y_{2}=\lambda y_{1}, \\ 
D_{q}y_{1}=\lambda y_{2},%
\end{array}%
\right.  \tag{4.1}
\end{equation}%
\begin{equation}
y_{1}\left( 0\right) =0,  \tag{4.2}
\end{equation}%
\begin{equation}
y_{2}\left( \pi q^{-1}\right) =0.  \tag{4.3}
\end{equation}%
It is easy to see that a solution (4.1) and (4.2) is given by 
\begin{equation*}
\phi ^{\top }\left( x,\lambda \right) =\left( \sin \left( \lambda x;q\right)
,~\cos \left( \lambda \sqrt{q}x;q\right) \right) .
\end{equation*}
By substituting this solution in (4.3), we obtain $\omega \left( \lambda
\right) =\cos \left( \lambda q^{-1\backslash 2}\pi ;q\right) ,$ hence, the
eigenvalues are $\lambda _{n}=\dfrac{q^{1-n+\varepsilon _{n}\left(
1\backslash 2\right) }}{\left( 1-q\right) \pi }.$ Applying Theorem 1, the $%
q- $transforms%
\begin{eqnarray}
F\left( \lambda \right) &=&\int\limits_{0}^{\pi }f^{\top }\left( x\right)
\phi \left( x,\lambda \right) d_{q}x  \notag \\
&=&\int\limits_{0}^{\pi }\left\{ f_{1}\left( x\right) \sin \left( \lambda
x;q\right) +f_{2}\left( x\right) \cos \left( \lambda \sqrt{q}x;q\right)
\right\} d_{q}x,  \TCItag{4.4}
\end{eqnarray}%
for some $f_{1}$ and $f_{2}\in L_{q}^{2}\left( 0,\pi \right) ,$ then it has
the sampling formula%
\begin{equation}
F\left( \lambda \right) =\sum\limits_{n=-\infty }^{\infty }F\left( \lambda
_{n}\right) \frac{\cos \left( \lambda q^{-1\backslash 2}\pi ;q\right) }{%
\left( \lambda -\lambda _{n}\right) \omega ^{\prime }\left( \lambda
_{n}\right) }.  \tag{4.5}
\end{equation}

\textbf{Example 2.} Consider $q-$Dirac equation (4.1) together with the
following boundary conditions%
\begin{equation}
y_{2}\left( 0\right) =0,  \tag{4.6}
\end{equation}%
\begin{equation}
y_{1}\left( \pi \right) =0.  \tag{4.7}
\end{equation}%
In this case $\phi ^{\top }\left( x,\lambda \right) =\left( \cos \left(
\lambda x;q\right) ,~-\sqrt{q}\sin \left( \lambda \sqrt{q}x;q\right) \right)
.$ Since $\omega \left( \lambda \right) =\cos \left( \lambda \pi ;q\right) ,$
then the eigenvalues are given by $\lambda _{n}=\dfrac{q^{-n+1\backslash
2+\varepsilon _{n}\left( 1\backslash 2\right) }}{\left( 1-q\right) \pi }.$
Applying Theorem 1 above to the $q-$transform%
\begin{equation}
F\left( \lambda \right) =\int\limits_{0}^{\pi }\left\{ f_{1}\left( x\right)
\cos \left( \lambda x;q\right) -f_{2}\left( x\right) \sqrt{q}\sin \left(
\lambda \sqrt{q}x;q\right) \right\} d_{q}x,  \tag{4.8}
\end{equation}%
for some $f_{1}$ and $f_{2}\in L_{q}^{2}\left( 0,\pi \right) ,$ then we
obtain%
\begin{equation}
F\left( \lambda \right) =\sum\limits_{n=-\infty }^{\infty }F\left( \lambda
_{n}\right) \frac{\cos \left( \lambda \pi ;q\right) }{\left( \lambda
-\lambda _{n}\right) \omega ^{\prime }\left( \lambda _{n}\right) }. 
\tag{4.9}
\end{equation}

\textbf{Example 3. }Consider $q-$Dirac equation (4.1) together with the
following boundary conditions%
\begin{equation}
y_{1}\left( 0\right) +y_{2}\left( 0\right) =0,  \tag{4.10}
\end{equation}%
\begin{equation}
y_{2}\left( \pi q^{-1}\right) =0.  \tag{4.11}
\end{equation}%
In this case 
\begin{equation*}
\phi ^{\top }\left( x,\lambda \right) =\left( \cos \left( \lambda x;q\right)
-\sin \left( \lambda x;q\right) ,~-\sqrt{q}\sin \left( \lambda \sqrt{q}%
x;q\right) -\cos \left( \lambda \sqrt{q}x;q\right) \right) .
\end{equation*}
Since $\omega \left( \lambda \right) =-\sqrt{q}\sin \left( \lambda
q^{-1\backslash 2}\pi ;q\right) -\cos \left( \lambda q^{-1\backslash 2}\pi
;q\right) ,$ then the eigenvalues of this problem are the solutions of
equation%
\begin{equation}
\sqrt{q}\sin \left( \lambda q^{-1\backslash 2}\pi ;q\right) =-\cos \left(
\lambda q^{-1\backslash 2}\pi ;q\right) .  \tag{4.12}
\end{equation}
Applying Theorem 1 above to the $q-$transform%
\begin{eqnarray}
F\left( \lambda \right) &=&\int\limits_{0}^{\pi }\left\{ f_{1}\left(
x\right) \left( \cos \left( \lambda x;q\right) -\sin \left( \lambda
x;q\right) \right) \right.  \notag \\
&&\left. -f_{2}\left( x\right) \left( \sqrt{q}\sin \left( \lambda \sqrt{q}%
x;q\right) +\cos \left( \lambda \sqrt{q}x;q\right) \right) \right\} d_{q}x, 
\TCItag{4.13}
\end{eqnarray}%
for some $f_{1}$ and $f_{2}\in L_{q}^{2}\left( 0,\pi \right) ,$ then we
obtain%
\begin{equation}
F\left( \lambda \right) =\sum\limits_{n=-\infty }^{\infty }F\left( \lambda
_{n}\right) \frac{-\sqrt{q}\sin \left( \lambda q^{-1\backslash 2}\pi
;q\right) -\cos \left( \lambda q^{-1\backslash 2}\pi ;q\right) }{\left(
\lambda -\lambda _{n}\right) \omega ^{\prime }\left( \lambda _{n}\right) }. 
\tag{4.14}
\end{equation}


\begin{thebibliography}{99}
\bibitem{1} {\small Allahverdiev, B.P., Tuna, H., One-dimensional q-Dirac
equation, Math. Meth. Appl. Sci., 40, 7287--7306 (2017)}

\bibitem{2} {\small H\i ra, F., Eigenvalues and eigenfunctions of q-Dirac
system, 2018, preprint, arXiv:submit/2231201.}

\bibitem{3} {\small Allahverdiev, B.P., Tuna, H.,Dissipative q-Dirac
operator with general boundary conditions, Quaest. Math., 1-17, (2017)}

\bibitem{4} {\small Jackson, F.H., On q-definite integrals, Q. J. Pure Appl.
Math., 41, 193-203 (1910)}

\bibitem{5} {\small Gasper, G., Rahman, M., Basic Hypergeometric Series,
Cambridge Univ. Press, New York (1990)}

\bibitem{6} {\small Kac, V., Cheung, P., Quantum Calculus, Springer, New
York (2002)}

\bibitem{7} {\small Bangerezako, G., An Introduction to q-Difference
Equations, preprint (2008)}

\bibitem{8} {\small Annaby, M.H., Mansour, Z.S., q-Fractional Calculus and
Equations, Springer, 2056 (2012)}

\bibitem{9} {\small Andrews, G.E., Askey, R., Roy, R., Special Functions,
Cambridge Univ. Press, Cambridge (1999)}

\bibitem{10} {\small Annaby, M.H., Mansour, Z.S., On the zeros of basic
finite Hankel transforms, J. Math. Anal. Appl., 323, 1091-1103 (2006)}

\bibitem{11} {\small Annaby, M.H., Mansour, Z.S., A basic analog of a
theorem of P\'{o}lya, Math. Z., 258, 363-379 (2008)}

\bibitem{12} {\small Annaby, M.H., q-type sampling theorems, Result. Math.,
44, 214-225 (2003)}

\bibitem{13} {\small Annaby, M.H., Bustoz, J., Ismail, M.E.H., On sampling
theory and basic Sturm-Liouville systems, J. Comput. Appl. Math., 206, 73-85
(2007)}

\bibitem{14} {\small Levitan, B.M., Sargsjan, I.S., Sturm-Liouville and
Dirac Operators, Kluwer, Dordrecht (1991)}

\bibitem{15} {\small Annaby, M.H., Mansour, Z.S., Basic Sturm-Liouville
problems, J. Phys. A:Math. Gen., 38, 3775-3797 (2005)}

\bibitem{16} {\small Annaby, M.H., Mansour, Z.S., Asymptotic formulae for
eigenvalues and eigenfunctions of q-Sturm-Liouville problems, Math. Nachr.,
284, 443-470 (2011)}

\bibitem{17} {\small Whittaker, E., On the functions which are represented
by the expansion of the interpolation theory, Proc. Roy. Soc. Edinburgh
Sect., A 35, 181--194 (1915) }

\bibitem{18} {\small Kotel'nikov, \ V., On the carrying capacity of the
\textquotedblleft ether\textquotedblright\ and wire in telecommunications,
in: Material for the All-Union Conference on Questions, Izd. Red. Upr.
Svyazi RKKA, Moscow (1933)}

\bibitem{19} {\small Shannon, C.E., Communications in the presence of noise,
Proc. IRE 37, 10--21 (1949)}

\bibitem{20} {\small Weiss, P., Sampling theorems associated with
Sturm--Liouville systems, Bull. Amer. Math. Soc., 163, 242 (1957)}

\bibitem{21} {\small Kramer, H.P., A generalized sampling theorem, J. Math.
Phys., 38, 68--72 (1959)}

\bibitem{22} {\small Zayed, A.I., Garc\'{\i}a, G.A., Sampling theorem
associated with a Dirac operator and the Hartley transform, J. Math. Anal.
Appl., 214, 587-598 (1997)}

\bibitem{23} {\small Ismail, M.E.H., Zayed, A.I., A q-analogue of the
Whittaker-Shannon-Kotel'nikov sampling theorem, Proc. Amer. Math. Soc., 131,
3711-3719 (2003)}

\bibitem{24} {\small Abreu, L.D., Sampling theory associated with
q-difference equations of the Sturm-Liouville type, J. Phys. A:Math. Gen.,
38, 10311-10319 (2005)}

\bibitem{25} {\small Annaby, M.H., Hassan, H.A., Mansour, Z.S., Sampling
theorems associated with singular q-Sturm-Liouville problems, Results.
Math., 62, 121-136 (2012)}

\bibitem{26} {\small Annaby, M.A., Hassan, H.A., Sampling theorems for
Jackson-N\"{o}rlund transforms associated with Hahn-difference operators, J.
Math. Anal. Appl. (2018), https://doi.org/10.1016/j.jmaa.2018.04.016 }
\end{thebibliography}
\end{document}